\documentclass{amsart}

\usepackage{amsthm, amsfonts, amssymb, amsmath, latexsym, enumerate, times}
\usepackage[latin1]{inputenc}
\usepackage{mathrsfs}
\usepackage{array}
\usepackage{comment}
\usepackage{paralist}
\usepackage{color}
\newtheorem{theorem}{Theorem}
\newtheorem{lemma}[theorem]{Lemma}

\newtheorem{corollary}[theorem]{Corollary}
\newtheorem{proposition}[theorem]{Proposition}

\theoremstyle{definition}

\newtheorem{remark}[theorem]{Remark}

\newcommand{\Pic}{{\rm Pic}}

\renewcommand{\L}{{\mathcal L}}
\newcommand{\R}{{\mathbb R}}
\newcommand{\bbZ}{{\mathbb Z}}

\renewcommand{\L}{{\mathcal L}}

\def\geq{\geqslant}
\def\leq{\leqslant}

\begin{document}
 
\title[Non effective planar  linear systems]{Non effective planar linear systems at the boundary of the Mori cone}

\author{Ciro Ciliberto}
\address{Dipartimento di Matematica, Universit\`a di Roma Tor Vergata, Via O. Raimondo
 00173 Roma, Italia}
\email{cilibert@axp.mat.uniroma2.it}

%author two
\author{Rick Miranda}
\address{Department of Mathematics, Colorado State University, Fort Collins (CO), 80523,USA}
\email{rick.miranda@colostate.edu}
 
%\subjclass{Primary ???; Secondary ???}
 
 
\maketitle

\centerline{We dedicate this paper to the memory of Alberto Collino}

\begin{abstract} 
In this paper we prove that certain linear systems 
(and all their multiples)  of plane curves 
with general base points and zero--self intersection are empty, 
thus exhibiting further examples of rays 
at the boundary of the Mori cone of a general blow--up of the plane.
\end{abstract}

\section*{Introduction} 

Let $X_n$ be the blow--up of the projective plane at $n$ general points. Let $\L_d(m_1,\ldots, m_n)$, with $d>0$, be the linear system on $X_n$ corresponding to plane curves of degree $d$ with general points of multiplicities at least $m_1\geq \cdots \geq m_n$ (we will use exponential notation for repeated multiplicities). 

We assume $n\geq 3$. We define 
$$
N=\# \{j\;|\;m_j\geq2\}, \quad h=\# \{j\;|\;m_j=1\}
$$
so that $N+h\leq n$. 

Let us make the hypothesis $\mathcal H$: \\
\begin{inparaenum}
\item [(i)] $d \geq m_1 \geq m_2\geq  \cdots \geq m_n \geq 0$, \\
\item [(ii)] $e=d -( m_1+m_2+m_3)\geq 0$,\\
\item [(iii)]  $d^2=\sum_{i=1}^nm_i^2$.
%\item [(iv)] if $N=8$ then $\mathcal{L}_d(m_1,\ldots,m_n)$ is different from a multiple of $\mathcal{L}_6(2^8,1^4)$.
\end{inparaenum}

Note that  condition (ii) implies that $\L_d(m_1,\ldots, m_n)$ is \emph{Cremona reduced}, i.e., its degree $d$ cannot be reduced by applying quadratic transformations based at its assigned base points.

Our goal in this article is to prove the following:

\begin{theorem}\label{thm:main} 
Suppose that $\mathcal{L}_d(m_1,\ldots,m_n)$
is a linear system satisfying the hypothesis $\mathcal H$,
for which $N\leq 8$.
Then for any $k \geq 1$, the system $\mathcal{L}_{kd}(km_1, \ldots,km_n)$ is empty, unless:\\
\begin{inparaenum}
\item [(a)] $\mathcal{L}_d(m_1,\ldots,m_n)$ is a multiple of $\mathcal L_1(1)$;\\
\item [(b)] $\mathcal{L}_d(m_1,\ldots,m_n)$ is a multiple of $ \mathcal L_3(1^9)$.
\end{inparaenum}
\end{theorem}

If we set $N_1(X_{n})=\Pic(X_{n})\otimes _\bbZ \R$, 
then any system $\mathcal{L}_d(m_1,\ldots,m_n)$
such that $\mathcal{L}_{kd}(km_1, \ldots,km_n)$ is empty for any $k \geq 1$
determines a rational ray in $N_1(X_{n})$ that is not effective (see \cite [\S 3.1] {CHMR}). Therefore such a ray 
sits in the boundary of the Mori cone of $X_{n}$. 
Any such ray, if rational, is called a \emph{good ray} in \cite [\S 3.2] {CHMR} 
whereas, if irrational, it is called a \emph{wonderful ray}.
No wonderful ray has been discovered up to now. 
Proving that a given ray is good seems in general to be difficult, 
and in \cite {CHMR} the authors were able to exhibit some examples. 
Other examples have been provided in \cite {CM21} 
and they correspond to the case $N=1$ in Theorem \ref{thm:main}.

Our proof is inductive on $N$,
and uses the degeneration introduced in \cite {CM98};
we will recall the main lemma that provides the basic reduction step in \S \ref {sec:ind} 
(see Lemma \ref {lem:two}). 
In \S  \ref {sec:deg6} we will separately treat the cases 
for which our general inductive strategy fails. 
The proof of  Theorem \ref{thm:main} is in \S \ref {sec:proof}. 

\medskip

{\bf Acknowledgements:} Ciro Ciliberto is a member of GNSAGA of INdAM. 

\section{The cases $\mathcal{L}_6(2^8,1^4)$ and $\mathcal{L}_6(2^7,1^8)$}\label{sec:deg6} 

In this section we prove the:

\begin{proposition}\label{prop:6} For any positive integer $k$ the linear system $\L=\mathcal{L}_{6k}((2k)^8,k^4)$ is empty.
\end{proposition}

\begin{proof} We use the collision technique introduced in \cite{CM05}, specifically the four points collision stated there in Proposition 3.1(c), which says that a general collision of four points of multiplicity $k$ results in a point of multiplicity $2k$ with a matching condition, i.e., the $2k$ tangent directions on the exceptional $\mathbb P^1$ are invariant under an involution which can be taken to be general. 

We apply this  to the four $k$--tuple points of $\L$ and conclude that the emptiness of $\L$ follows from the emptiness of the subsystem of $\mathcal{L}_{6k}((2k)^9)$ which satisfies the aforementioned matching condition on the ninth point. However   $\mathcal{L}_{6k}((2k)^9)$ has dimension zero, consisting of the unique cubic through the nine points with multiplicity $2k$, and this curve does not satisfy the matching condition. 
\end{proof}

\begin{corollary}\label{prop:6.1} For any positive integer $k$ the linear system $\L=\mathcal{L}_{6k}((2k)^7,k^8)$ is empty.
\end{corollary}

\begin{proof} We use again the collision technique as above, colliding four of the points of multiplicity $k$ to a point of multiplicity $2k$. Then the emptiness of $\L$ follows from the emptiness of the system of $\mathcal{L}_{6k}((2k)^8, k^4)$, proved in Proposition \ref {prop:6}.\end{proof}

These two results are contained in 
\cite[Remark 5.1.6, Remark 5.5.11 and Proposition 5.5.10]{CHMR}; 
we have included the brief proofs here for completeness.

\section{The inductive strategy}\label{sec:ind}

In this section we will present the strategy of the proof of Theorem \ref {thm:main} and we will prove a couple of useful lemmas. 

The proof will be by induction on $N$ and we first prove the statement for $N=0,1$. Then, for $N\geq 2$, we start with a given system  $\mathcal{L}_d(m_1,\ldots,m_n)$ satisfying $\mathcal H$ and different from multiples of the forbidden systems (a) and (b) of Theorem \ref {thm:main}. Since $\mathcal{L}_d(m_1,\ldots,m_n)$ satisfies $\mathcal H$, we have $e\geq 0$. The first step is to reduce to $e=0$. For this we use Lemma \ref {lem:two} below, which raises $m_1$ by 1,  decreases  $e$ by 1 (and preserves $\mathcal H$). Repeated applications of this allow us to assume $e=0$. 

Next we define a \emph{basic move} on $e=0$ systems that satisfy $\mathcal H$. Apply again  Lemma \ref {lem:two}  and increase $m_1$ by $1$. Since $e=0$, the resulting system is no longer Cremona reduced, and we reduce it: the hypotheses allow us to control the applications of the needed quadratic transformations. Then, if $e>0$ for the new system, we apply Lemma \ref {lem:two} again repeatedly to obtain $e=0$. We note that in this process $\mathcal H$ is always preserved.  

The strategy is now to apply basic moves iteratively and show that either $N$ must decrease, so that we can apply induction, or $N=8$ and we arrive at  a system that we directly know is empty, such as a  multiple of the degree 6 linear systems considered in \S \ref {sec:deg6}. 

Next we present the aforementioned useful lemmas.

\begin{lemma}\label{lem:one} Suppose $1 \leq N \leq 8$, $\mathcal H$ holds and the system is not $\mathcal{L}_6(2^8,1^4)$. Then $h \geq 2m_1+1$, unless $N=1$ and $d=m_1$ in which case $h=0$.
\end{lemma}

\begin{proof} 
If $N=1$ and $d>m_1$, then $h=d^2-m_1^2\geq (m_1+1)^2-m^2_1= 2m_1+1$ and we are done.

Suppose $N=2$. Then $d \geq m_1+m_2$, so 
$h=d^2-m_1^2-m^2_2\geq 2m_1m_2\geq 4m_1\geq 2m_1+1$ as wanted.

Next suppose $N\geq 3$.
%\color{red}{It suffices to prove the statement for $N=8$} (this sentence should be deleted, it is useless). \color{black}
Since $d \geq m_1+m_2+m_3$, we have 
$$d^2 \geq (m_1+m_2+m_3)^2 = m_1^2+m_2^2+m_3^2 + 2m_1m_2+2m_1m_3+2m_2m_3.$$
As the $m_i$s are in descending order, we have $2m_2m_3 \geq m_4^2+m_5^2$, $2m_1m_3 \geq m_6^2+m_7^2$, and $m_1m_2 \geq m_8^2$, so that $h \geq m_1m_2 \geq 2m_1$.

Suppose that $h=2m_1$. Then all the above inequalities are equalities, which implies $N=8$ and all $m_i$s equal to $2$.
In that case $h = d^2 - 32$, forcing $d=6$, which is forbidden by hypothesis. Hence $h \geq 2m_1+1$. \end{proof}

Suppose next $2\leq N \leq 8$, $\mathcal{L}_d(m_1,\ldots,m_n)$ satisfies $\mathcal H$ and it is different from $\mathcal{L}_6(2^8,1^4)$.

Fix $k \geq 1$ and consider $\mathcal{L}_{kd}(km_1, \ldots,km_n)$. 
Then we can make the $P$--$F$ degeneration described in \cite {CM98} with the limit line bundle of $\mathcal{L}_{kd}(km_1, \ldots,km_n)$ having aspects 
$$\mathcal{L}_P = \mathcal{L}_{k(m_1+1)}(km_1, k^a)$$ 
and
$$\mathcal{L}_F = \mathcal{L}_{kd}(k(m_1+1),km_2,\ldots,km_N, k^b)$$
with $a = 2m_1+1$ and $b = h-2m_1-1$ on $P$ and $F$ respectively. Note that Lemma \ref {lem:one} implies that $a$ and $b$ are positive which is the prerequisite to apply the $P$--$F$ degeneration described above. 

\begin{lemma}\label{lem:two}
Suppose that, in the above setting, $\mathcal{L}_F$ is empty. Then so is $\mathcal{L}_{kd}(km_1, \ldots,km_n)$.
\end{lemma}

\begin{proof} First we notice that $\mathcal{L}_{m_1+1}(m_1, 1^a)$ is a pencil and no curve in this pencil contains a general line.
Then  $\mathcal{L}_{k(m_1+1)}(km_1, k^a)$ is composed with this pencil and no curve in this system contains a general line. Hence the system 
$\hat {\mathcal{L}}_P=\mathcal{L}_{km_1+k-1}(km_1, k^a)$ is empty. 
Then the analysis in \cite {CM98} implies the result, because a section of the limit line bundle is zero on $F$ by hypothesis, and is also zero on $P$ by the emptiness of $\hat {\mathcal{L}}_P$. \end{proof}

%\color{red}   \begin{remark}\label{rem:basic} The consequence of Lemmas \ref {lem:one} and \ref{lem:two} is that if a linear system satisfies $\mathcal H$ and does not coincide with $\mathcal{L}_6(2^8,1^4)$, then the basic move on any multiple of it is available and results in a linear system which is a multiple of a system with lower multiplicities. 
%\end{remark}
%(Rick, I find this remark useless and actually confusing; I would delete it)\color{black}

An additional application of the $P$--$F$ degeneration method enables us to prove the following statements.  The first one is also a consequence of
\cite[Remark 5.5.11]{CHMR}, where more general statements are made; 
we include the brief proof for the convenience of the reader.

\begin{proposition}\label{prop:9} For any positive integer $k$ the linear system $\L=\mathcal{L}_{9k}((3k)^8,k^9)$ is empty.
\end{proposition}

\begin{proof}  
We make the $P$--$F$ degeneration as above. 
The relevant systems are  
$$
\L_F=\L_{9k}(4k, (3k)^7, k^2)\quad \text{and}\quad   \L_P=\L_{4k}(3k, k^7).
$$
By Cremona reducing $\L_F$ we see that it consists of 
a unique curve with multiplicity $k$
which is the Cremona image of a cubic through $9$ general simple points. 
It meets the double curve $R=P\cap F$ at $4$ points with multiplicity $k$ 
which can be assumed to be general. 
The kernel system $\hat{\L}_F$ is empty. 
The system $\L_P$ is composed with the pencil $\L_{4}(3, 1^7)$, 
and it cuts out on $R$ the linear system (a $g^k_{4k}$)
composed with the $g^1_4$ cut out on $R$ by $\L_{4}(3, 1^7)$. 
The kernel linear system $\hat{\L}_P$ is also empty. 
By the generality of the restriction of $\L_F$ to $R$, 
there can be no matching divisor in $\L_P$ with the unique curve in $\L_F$ 
and therefore no section of the limit line bundle, since both kernel linear systems are empty (see \cite {CM98}). This ends the proof of the proposition. 
\end{proof}

Before the next lemma we need a simple remark.

\begin{remark}\label{rem:simple} We notice that a linear system $\mathcal{L}_d(m_1,\ldots,m_n)$ cannot be a multiple of the system (a) of Theorem \ref {thm:main} if $m_2\geq 1$ and it cannot be a multiple of the system (b) of Theorem \ref {thm:main} if  $m_1>m_8$. \end{remark}

\begin{lemma}\label{lem:three} Let $\mathcal{L}_d(m_1,\ldots,m_n)$ satisfy $\mathcal H$, with $8\geq N\geq 2$ and $e=0$ and assume it is different from a multiple of $\mathcal{L}_6(2^8,1^4)$ or of $\L_9(3^8, 1^9)$ or of $\L_6(2^7, 1^8)$.  
Then the result of a basic move is different from a multiple of either of the two forbidden systems (a) and (b) of Theorem \ref {thm:main}.
\end{lemma}

\begin{proof}  Since the system is not a multiple of $\mathcal{L}_6(2^8,1^4)$ we can do a basic move (see Lemma \ref {lem:one}). 

The first step in the basic move leads to the system  $\mathcal{L}_d(m_1+1,m_2,\ldots,m_n)$. Then, applying the quadratic transformation based at the first three points, gives the linear system 
$\L_1:=\mathcal{L}_{d-1}(m_1,m_2-1,m_3-1,m_4\ldots,m_n)$. 

If $m_3>m_4$, the three highest multiplicities are $m_1,m_2-1,m_3-1$ and therefore the system is Cremona reduced, with $e=1$.  At this point, to finish the basic move, we make one more application of Lemma \ref {lem:two} to reduce to the case $e=0$. This leads to the system $\mathcal{L}_{d-1}(m_1+1,m_2-1,m_3-1,m_4,\ldots,m_n)$, and we conclude by applying Remark \ref {rem:simple}. 

If $m_3=m_4$ but $m_2>m_5$ then the three highest multiplicities of $\L_1$ are $m_1, m_2-1, m_4=m_3$ in some order.  We see that $e=0$ for $\L_1$ so the basic move is finished and we conclude again by using Remark \ref {rem:simple}. 

We may now assume $m_3=m_4$ and $m_2=m_5$, hence $m_2=\cdots=m_5=m$. In this case reordering the multiplicities of $\L_1$ the three highest ones are $m_1,m,m$. This is not Cremona reduced. After making the quadratic transformation based at the three points of highest multiplicity, we get to the system $\L_2=\L_{d-2}(m_1-1,(m-1)^4, m_6,\ldots, m_n)$. Now we have three cases.

The first case is $m_6<m_2=m$. In that case the three highest multiplicities of $\L_2$ are $m_1-1, m-1, m-1$, the system is Cremona reduced and $e=1$. Moreover 
 $\L_2$ does not coincide with a multiple of $\mathcal{L}_6(2^8,1^4)$ that all have $e=0$.  Then we can apply Lemma \ref {lem:two} one more time to finish the basic move and we get the system $\L_{d-2}(m_1,(m-1)^4, m_6,\ldots, m_n)$. We conclude again by using Remark \ref {rem:simple} because $m_1>m-1\geq 1$. 
 
In the second case we have $m_6=m$ and $m_7<m$. In that case the three highest multiplicities of $\L_2$ are $m_1-1, m, m-1$ and the system is Cremona reduced with  $e=0$. This ends the basic move and we apply Remark \ref{rem:simple} to finish. 

In the final case we have $m_6=m_7=m$. Now the three highest multiplicities of $\L_2$ are $m_1-1, m, m$ and the system is not Cremona reduced. One more quadratic transformation gives $\L_3=\L_{d-3}(m_1-2,(m-1)^6, m_8,\ldots, m_n)$. 

If $m_1>m$ and $m_8<m$ the three highest multiplicities are $m_1-2, m-1, m-1$, the system is Cremona reduced, with $e=1$. In this case we have to make a further step to accomplish the basic move and we get $\L_{d-3}(m_1-1,(m-1)^6, m_8,\ldots, m_n)$. We apply again Remark \ref{rem:simple} to finish.

If $m_1>m$ and $m_8=m$ the three highest multiplicities of $\L_3$ are $m_1-2, m, m-1$. Then $\L_3$  is Cremona reduced with $e=0$, the basic move is finished and  we conclude by applying Remark \ref{rem:simple} since $m>m-1$.

 If $m_1=m$ and  $m_8<m$, the three highest multiplicities of $\L_3$ are $m-1, m-1, m-1$. Then $\L_3$ is Cremona reduced with $e=0$. The basic move is finished and   by applying Remark \ref{rem:simple} we see that this is not a multiple of the forbidden system (a) of Theorem \ref {thm:main} by Remark \ref {rem:simple}. It could be a multiple of the (b) system, but this only happens if $m=2$, $d=6$, and the original system is $\L_6(2^7,1^8)$, which is forbidden by hypothesis.
 
 If $m_1=m$ and $m_8=m$ the system we start with is $\L_{3m}(m^8, 1^{m^2})$ and we notice that $m\geq 3$ because $m=2$ is forbidden by hypothesis. Then in $\L_3$  the three highest multiplicities are $m, m-1, m-1$. At this point  the system is not Cremona reduced. Reducing it one gets $\L_4:=\L_{d-6}((m-2)^7,m-3, m_9,\ldots, m_n)$.  This again is not a multiple of the forbidden system (a) of Theorem \ref {thm:main}. It could be a multiple of the (b) system, but this only happens if $m=3$, $d=9$, and the original system is $\L_9(3^8,1^9)$, which is forbidden by hypothesis. 
 
 This ends the proof of the lemma. \end{proof}

\section{The proof of Theorem \ref{thm:main}}\label{sec:proof}
Now we are in a position to prove Theorem \ref{thm:main}. 

 \begin{proof}[Proof of Theorem \ref{thm:main}] 
If $N=0$, then $\L = \L_d(1^{d^2})$. 
We have excluded the $d=1$ case. 
The $d=2$ case does not verify $\mathcal H$. 
We have excluded the $d=3$ case. 
For $d \geq 4$, this is Nagata's Theorem (see \cite {N}).

If $N=1$, then $\L = \L_d(m_1, 1^h)$ where $h = d^2-m_1^2$. 
If $h=0$ then this is a multiple of $\L_1(1)$, which we have excluded. 
If $h \geq 1$ then by Lemma \ref {lem:one}, we have 
$h \geq 2m_1+1$, so that there are at least five simple points.  z
Since $\mathcal H$ holds, we must have $d \geq m_1+1+1 \geq 4$. 
Then the result is contained in \cite {CM21}.

Next we assume $N\geq 2$. 
By Propositions \ref {prop:6}, \ref {prop:9}, and Corollary \ref {prop:6.1}, 
we can assume that the system is not a multiple of $\mathcal{L}_6(2^8,1^4)$ 
or of $\L_9(3^8, 1^9)$ or of $\L_6(2^7, 1^8)$. 
By applying (if necessary) Lemma \ref {lem:two}, we can assume that $e=0$. 
Then the hypotheses of Lemma \ref {lem:three} are met 
and we can make a basic move which does not arrive at a forbidden system. 
If the result of the move is a multiple of $\mathcal{L}_6(2^8,1^4)$
or of $\L_9(3^8, 1^9)$ or of $\L_6(2^7, 1^8)$, 
that are empty by Propositions \ref {prop:6}, \ref {prop:9}, and Corollary \ref {prop:6.1}, 
we apply Lemma \ref {lem:two} again and conclude that the original system is empty. 
If not we have reduced to a linear system with lower multiplicities.
%(see Remark \ref {rem:basic}). 
So, repeated applications of a basic move 
either results in an empty system or eventually decreases $N$, 
so that we can finish by induction. 
\end{proof}

We notice that the theorem is certainly false for the forbidden systems (a) and (b) of Theorem \ref {thm:main}.

\end{document}